\documentclass{amsart}

\usepackage{amsfonts}
\usepackage{amssymb}
\usepackage{enumerate}
\usepackage{amsmath}
\usepackage{amsthm}
\usepackage{graphicx}

\numberwithin{equation}{section}

\newtheorem{theorem}{Theorem}[section]
\newtheorem{lemma}[theorem]{Lemma}

\newtheorem{question}[theorem]{Question}

\theoremstyle{definition}

\newtheorem{remark}[theorem]{Remark}

\DeclareMathOperator{\diam}{diam}

\DeclareMathOperator{\length}{length}

\newtheorem*{xx}{Theorem \ref{thm}}

\begin{document}

\title{Intersection of continua and rectifiable curves}

\author{Rich\'ard Balka}
\address{Alfr\'ed R\'enyi Institute of Mathematics,
PO Box 127, 1364 Budapest, Hungary}
\email{balka.richard@renyi.mta.hu}
\thanks{We gratefully acknowledge the support of the
Hungarian Scientific Research Fund grant no.~72655.}

\author{Viktor Harangi}
\address{Alfr\'ed R\'enyi Institute of Mathematics,
PO Box 127, 1364 Budapest, Hungary}
\email{harangi.viktor@renyi.mta.hu}

\date{}

\begin{abstract}
We prove that for any non-degenerate continuum $K \subseteq \mathbb{R}^d$
there exists a rectifiable curve such that
its intersection with $K$ has Hausdorff dimension $1$.
This answers a question of B.~Kirchheim.
\end{abstract}



\maketitle

\section{Introduction}

A topological space $K$ is called a \emph{continuum}
if it is compact and connected.
The following question was asked by B.~Kirchheim \cite{BK}.

\begin{question} \label{q}
Does there exist a non-degenerate curve (or more generally, a continuum)
$K\subseteq \mathbb{R}^{d}$ such that every rectifiable curve intersects $K$
in a set of Hausdorff dimension less than $1$?
\end{question}

The motivation behind this question was that
in \cite[Examples (b), p. 208.]{G}
Gromov implicitly suggested that such curves exist.
In this paper we answer Question \ref{q} in the negative.

\begin{xx}
For any non-degenerate continuum $K\subseteq \mathbb{R}^{d}$ there exists
a rectifiable curve such that its intersection with $K$
has Hausdorff dimension $1$.
\end{xx}

\begin{remark}
Finding a $1$-dimensional intersection is the best we can hope for,
since any purely unrectifiable curve $K$ in the plane
(e.g., the Koch snowflake curve) has the property that
the intersection of $K$ and a rectifiable curve has zero $\mathcal{H}^{1}$ measure.
\end{remark}

%In Section~\ref{s:prelim} we recall some definitions and facts that we use
%in this paper. Finally, in Section~\ref{s:proof} we prove Theorem~\ref{thm}.

\section{Preliminaries} \label{s:prelim}

The diameter and the boundary of a set $A$ are denoted by
$\diam A$ and $\partial A$, respectively.
For $A\subseteq \mathbb{R}^{d}$ and $s \ge 0$ the
\emph{$s$-dimensional Hausdorff measure} is defined as
%%%%
\begin{align*}
\mathcal{H}^{s}(A)&=\lim_{\delta\to 0+}\mathcal{H}^{s}_{\delta}(A)
\mbox{, where}\\
\mathcal{H}^{s}_{\delta}(A)&=\inf \left\{ \sum_{i=1}^\infty (\diam
A_{i})^{s}: A \subseteq \bigcup_{i=1}^{\infty} A_{i},~
\forall i \diam A_i \le \delta \right\}.
\end{align*}
%%%%%
Then the \emph{Hausdorff dimension} of $A$ is
\[
\dim_{H} A = \sup\{s \ge 0: \mathcal{H}^{s}(A)>0\}.
\]
%%%%%%
Let $A\subseteq \mathbb{R}^{d}$ be non-empty and bounded, and let $\delta>0$. Set
%%%%%%
$$N(A,\delta)=\min\left \{k: A\subseteq \bigcup_{i=1}^{k} A_i,~
\forall i \ \diam A_i\leq \delta \right\}.$$
%%%%%
The \emph{upper Minkowski dimension} of $A$ is defined as
%%%%%%
$$\overline{\dim}_{M}(A)=\limsup_{\delta \to 0+}
\frac{\log N(A,\delta)}{-\log \delta}.$$
%%%%%%
If $A\subseteq \mathbb{R}^{d}$ is non-empty and bounded, then it follows
easily from the above definitions that
%%%%%%
$$\dim_{H}A\leq \overline{\dim}_{M}(A).$$
%%%%%%%
For more information on these concepts see \cite{F} or \cite{Ma}.

A continuous map $f\colon [a,b]\to \mathbb{R}^{d}$ is called a \emph{curve}.
Its \emph{length} is defined as
%%%%%
$$\length (f)=\sup \left\{\sum_{i=1}^{n} |f(x_{i})-f(x_{i-1})|:
n\in \mathbb{N}^{+},~ a=x_0<\dots<x_n=b \right\}.$$
%%%%%%
If $\length(f)<\infty$, then $f$ is said to be \emph{rectifiable}.
We say that $f$ is \emph{naturally parametrized}
if for all $x,y\in [a,b]$, $x\leq y$ we have
%%%%%
$$\length \left(f|_{[x,y]}\right)=|x-y|.$$
%%%%%%%
We simply write $\Gamma=f([a,b])$ instead of $f$
if the parametrization is obvious or not important for us.
For every non-degenerate rectifiable curve
$\Gamma$ we have $0<\mathcal{H}^{1} (\Gamma)<\infty$, so $\dim_{H} \Gamma=1$.
If $|f(x)-f(y)|\leq |x-y|$ for all $x,y\in [a,b]$,
then $f$ is called \emph{1-Lipschitz}.
Every naturally parametrized curve is clearly 1-Lipschitz.

\section{The proof} \label{s:proof}

First we need some lemmas. The following lemma is probably known,
but we could not find a reference, so we outline its proof.

\begin{lemma} \label{l:cover}
If a non-empty bounded set $A \subseteq \mathbb{R}^d$ has
upper Minkowski dimension less than $1$,
then a rectifiable curve covers $A$.
\end{lemma}

\begin{proof} We can assume that $A$ is compact and $A \subseteq [0,1]^d$,
since we can take its closure and transform it into the unit
cube with a similarity, this does not change
the upper Minkowski dimension of the set
and the fact whether it can be covered by a rectifiable curve.

For every $n\in \mathbb{N}$ we divide $[0,1]^d$ into non-overlapping
cubes with edge length $2^{-n}$ in the natural way,
and we denote the cubes that intersect $A$ by
%%%%%
$$ Q_{n, 1}, Q_{n, 2}, \ldots, Q_{n, r_n} ,$$
%%%%%%%
where $r_n$ is the number of such cubes.
As every set with diameter at most $2^{-n}$ can intersect
at most $3^d$ of the above cubes, we obtain $r_n\leq 3^d N(A,2^{-n})$.
Let us fix $s$ such that
$\overline{\dim}_{M}(A)<s<1$. By the definition of upper Minkowski
dimension there exists a constant $c_1\in \mathbb{R}$ such that for all $n\in
\mathbb{N}$
%%%%%%
\begin{equation} \label{rn}  r_n \leq c_1 \cdot 2^{sn}.
\end{equation}
%%%%%%%
Let $n\in \mathbb{N}$ and $i\in \{1,\dots,r_n\}$ be arbitrarily fixed.
Let $P_{n,i}$ be the vertex of $Q_{n,i}$ that is the closest to the origin.
If $Q_{n+1, j_1}, \ldots, Q_{n+1, j_m}$ are the next level
cubes contained by $Q_{n,i}$, then consider the broken line
%%%%
$$\Gamma_{n,i}=P_{n,i} P_{n+1, j_1} P_{n+1, j_2} \ldots P_{n+1,j_m} P_{n,i} .$$
%%%%%
Thus
%%%%%%
\begin{equation} \label{leng} \length(\Gamma_{n,i})\leq
(m+1) \diam Q_{n,i} \leq 2m \sqrt{d} 2^{-n}.
\end{equation}
%%%%
Let $l_n$ be the sum of these lengths for all $i\in \{1,\dots,r_n\}$.
Then \eqref{leng} and \eqref{rn} imply
%%%%%%
\begin{equation} \label{ln}
l_n\leq 2 r_{n+1} \sqrt{d} 2^{-n}
\leq 2c_1 \cdot 2^{s(n+1)}  \sqrt{d} 2^{-n} = c_2 2^{(s-1)n},
\end{equation}
%%%%%
where $c_2=c_1 \sqrt{d} 2^{s+1}$. We set
%%%%
$$L_n=\sum_{k=0}^{n}l_k \quad \textrm{and} \quad L=\sum_{k=0}^{\infty} l_k.$$
%%%%%
Since $s<1$, \eqref{ln} implies $L < \infty$.

Now we define the rectifiable curve covering $A$.
First we take the broken line $\Gamma_0=\Gamma_{0,1}$
with its natural parametrization $g_{0}\colon [0,L_0]\to \Gamma_0$.
Assume that the curves $g_{k}\colon [0,l_k]\to \Gamma_k$ are
already defined for all $k<n$. At every point $P_{n,i}$,
$i\in \{1,\dots, r_n\}$, we insert the broken line
$\Gamma_{n,i}$ in $\Gamma_{n-1}$, so we obtain
a naturally parametrized curve $g_{n}\colon [0,L_{n}]\to \Gamma_{n}$.

For every $n\in \mathbb{N}$ let us define $f_{n}\colon [0,L]\to \Gamma_n$ such that
$$f_{n}(x)= \begin{cases} g_{n}(x) & \textrm{ if } x\in [0,L_n], \\
g_{n}(L_n) & \textrm{ if } x\in [L_n,L].
\end{cases} $$
%%%%%%
Now we prove that the sequence $\langle f_n \rangle$ uniformly converges.
Let us fix $n\in \mathbb{N}$ and $x\in [0,L]$ arbitrarily.
As $\sum_{n=0}^{\infty} l_n<\infty$, it is enough to prove that
$|f_{n+1}(x)-f_{n}(x)|\leq l_{n+1}$. By construction
there exists $y\in [0,L]$ such that
$f_{n}(x)=f_{n+1}(y)$ and $|x-y| \leq l_{n+1}$.
Since $g_{n+1}$ is naturally parametrized, we obtain that
%%%%%%
$$|f_{n+1}(x)-f_{n}(x)|= |f_{n+1}(x)-f_{n+1}(y)|\leq |x-y|\leq l_{n+1}.$$
%%%%%
Therefore $\langle f_n \rangle$ uniformly converges to
some $f:[0,L] \to \mathbb{R}^d$. As a uniform limit of 1-Lipschitz functions
$f$ is also 1-Lipschitz, thus rectifiable.

It remains to prove that $A\subseteq f([0,L])$. Let $\vec{z} \in A$.
We need to show that there is $x\in [0,L]$ such that $f(x)=\vec{z}$.
For every $n\in \mathbb{N}$ there exists $i_n\in \{1,\dots, r_n\}$
such that $\vec{z} \in Q_{n,i_n}$. Let $x_n\in [0,L]$ such that
$f_{n}(x_n)=P_{n,i_n}$ for all $n\in \mathbb{N}$. By choosing a subsequence
we may assume that $x_{n}$ converges to some $x \in [0,L]$.
Therefore
%%%%
$$f(x)=\lim_{n\to \infty} f_{n} (x_{n})=
\lim_{n\to \infty} P_{n,i_{n}}=\vec{z}.$$
%%%%%
The proof is complete.
\end{proof}

The next lemma is \cite[Lemma 6.1.25]{E}.

\begin{lemma} \label{l:comp}
If $A$ is a closed subspace of a continuum $X$
such that $\emptyset \neq A\neq X$,
then for every connected component $C$ of $A$
we have $C\cap \partial A\neq \emptyset$.
\end{lemma}

We also need the following technical lemma.

\begin{lemma} \label{l:cube}
Suppose that $K\subseteq \mathbb{R}^{d}$ is a continuum
contained by a unit cube $Q$ and
$K$ has a point on each of two opposite sides of $Q$.
Then for any positive integer $N$
we can find $N$ pairwise non-overlapping cubes $Q_1, \ldots, Q_N$
with edge length $\frac 1N$ such that for each $i\in \{1,\dots, N\}$
there exists a continuum $K_i \subseteq K \cap Q_i$ with the property that
$K_i$ has a point on each of two opposite sides of $Q_i$.
\end{lemma}

\begin{proof}
Let $N\in \mathbb{N}^{+}$ be fixed. Set $S_0=\{0\}\times [0,1]^{d-1}$
and for all $i\in \{1,\dots, N\}$ consider
%%%%%%%
$$S_i=\{i/N\}\times [0,1]^{d-1} \quad \textrm{and} \quad
T_i=\left[(i-1)/N,i/N\right]\times [0,1]^{d-1}.$$
%%%%%
We may assume that $Q=[0,1]^{d}$ and
that the two opposite sides intersecting $K$
are $S_0$ and $S_N$. Let $\vec{x}\in K \cap S_0$
and $\vec{y} \in K \cap S_N$.

Now we prove that for each $i\in \{1,\dots ,N\}$
there is a continuum $C_i\subseteq K\cap T_i$ such that
$C_i\cap S_{i-1} \neq \emptyset$ and $C_i\cap S_i\neq \emptyset$.
Let $C_1$ be the component of $K\cap T_1$ containing $\vec{x}$.
Applying Lemma \ref{l:comp} for $X=K$, $A=K \cap T_1$, and $C=C_1$
yields that $C_1 \cap S_1 \neq \emptyset$. Let $C_2'$ be the component of
$K\cap \left( T_2\cup \dots \cup T_N\right)$ containing $\vec{y}$.
Similarly as above, we obtain $C_2'\cap S_1\neq \emptyset$.
If we continue this process, we get the required continua $C_2,\dots, C_N$.

Finally, for each $i\in \{1,\dots,N\}$ we construct
a cube $Q_i \subseteq T_i$ with edge length $\frac 1N$ and
a continuum $K_i\subseteq Q_i$ such that
$K_i$ has a point on each of two opposite sides of $Q_i$.
Clearly, the cubes $Q_i$ will be pairwise non-overlapping,
and it is enough to construct $Q_1$ and $K_1$
(one can get $Q_i, K_i$ similarly).
Let us consider the standard basis of $\mathbb{R}^d$:
$\vec{e}_1=(1,0,\dots,0),\dots,\vec{e}_d=(0,0,\dots,1)$.
Set $A_1=C_1$, $V_1=\{0\}\times \mathbb{R}^{d-1}$, $W_1=\{1/N\}\times \mathbb{R}^{d-1}$,
$Z_1=[0,1/N]\times \mathbb{R}^{d-1}$, and $m(1)=1$.
Then the definitions yield that $A_1$ has a point
on both $V_{m(1)}$ and $W_{m(1)}$.
Let $j\in \{2,\dots, d\}$ and assume that
$A_{k}$, $V_{k}$, $W_{k}$, $Z_{k}$, and $m(k)$
are already defined for all $k<j$ such that
$A_k$ has a point on both $V_{m(k)}$ and $W_{m(k)}$.
Let $\vec{x}_j\in A_{j-1}$ be a point which has minimal $j$th coordinate,
and let $V_j$ be the affine hyperplane that is orthogonal to
$\vec{e}_j$ and contains $\vec{x}_j$.
Set $W_j=V_j+\frac {1}{N} \vec{e}_j$,
and let $Z_j$ be the closed strip between $V_j$ and $W_j$.
If $A_{j-1}\subseteq Z_j$ then let $A_j=A_{j-1}$ and $m(j)=m(j-1)$.
If $A_{j-1} \nsubseteq Z_j$ then let $A_j$ be
the component of $\vec{x}_j$ in $A_{j-1}\cap Z_j$ and $m(j)=j$,
in this case Lemma \ref{l:comp} yields $A_{j}\cap W_j\neq \emptyset$.
Thus $A_j$ has a point on both $V_{m(j)}$ and  $W_{m(j)}$.
Let $Q_1=\bigcap _{j=1}^{d} Z_j$ and $K_1=A_d$.
Then $Q_1\subseteq S_1$ is a cube with edge length $\frac 1N$
and $K_1\subseteq Q_1$ is a continuum.
As $K_1$ has a point on both $V_{m(d)}$ and $W_{m(d)}$,
we obtain that $K_1$ has a point on each of two opposite sides of $Q_1$.
The proof is complete.
 \end{proof}

 Now we are ready to prove Theorem \ref{thm}.

\begin{theorem} \label{thm}
For any non-degenerate continuum $K\subseteq \mathbb{R}^{d}$ there exists
a rectifiable curve such that its intersection with $K$
has Hausdorff dimension $1$.
\end{theorem}

\begin{proof} By considering a similar copy of $K$
we may assume that $K$ is contained by a unit cube $Q$
and $K$ has a point on each of two opposite sides of $Q$.

Let $\varepsilon>0$ be arbitrary. First we prove the weaker result
that there exists $A\subseteq K$ such that
$ 1-\varepsilon \leq \dim_{H}A= \overline{\dim}_{M}(A)<1$.
By Lemma \ref{l:cover} $A$ is covered by a rectifiable curve.
Let us fix an integer $N\geq 2$ for which
$s:=\frac {\log (N-1)}{\log N} \geq 1-\varepsilon$.
We construct $A\subseteq K$ such that $\dim_{H}A=\overline{\dim}_{M}(A)=s$.
Set $\mathcal{I}_{n}=\{1,\dots, N-1\}^{n}$ for every $n\in \mathbb{N}^{+}$.
Iterating Lemma \ref{l:cube} implies that for all $n\in \mathbb{N}^{+}$ and
$(i_1,\dots, i_n)\in \mathcal{I}_{n}$
there are cubes $Q_{i_1 \dots i_{n}}$ in $Q$
with edge length $\frac{1}{N^{n}}$ such that
$Q_{i_1 \dots i_{n}}\subseteq Q_{i_1 \dots i_{n-1}}$,
and there are continua $K_{i_1 \dots i_{n}} \subseteq K$ such that
$K_{i_1 \dots i_{n}}\subseteq Q_{i_1 \dots i_{n}}\cap K_{i_1\dots i_{n-1}}$
and $K_{i_1 \dots i_{n}}$ has a point on each of
two opposite sides of $Q_{i_1 \dots i_{n}}$. Set
%%%%%%
\begin{equation*}
A_n=\bigcup_{i_1=1}^{N-1} \! \cdots \!
\bigcup_{i_n=1}^{N-1} K_{i_1 \dots i_{n}},
\end{equation*}
and let
\begin{equation*}
A=\bigcap_{n=1}^{\infty} A_n.
\end{equation*}
%%%%
Clearly, $A\subseteq K$ is compact.

On the one hand, as $A\subseteq A_n$ and $A_n$ is covered by $(N-1)^{n}$
many cubes of edge length $\frac{1}{N^{n}}$, we obtain that
$N(A_n, \sqrt{d}/N^{n})\leq (N-1)^{n}$ for all $n\in \mathbb{N}^{+}$.
Therefore $\overline{\dim}_{M}(A)\leq \frac {\log (N-1)}{\log N}=s$.

On the other hand, we prove that $\mathcal{H}^{s}(A)>0$.
Assume that $A\subseteq \bigcup_{j=1}^{\infty}  U_j$,
it is enough to prove that
$ \sum_{j=1}^{\infty} (\diam U_j)^{s}\geq \frac{1}{2^d(N-1)}$.
Clearly, we may assume that $U_j$ is a non-empty open set
with $\diam U_j < 1$ for each $j$,
and the compactness of $A$ implies that
there is a finite subcover $A\subseteq \bigcup_{j=1}^{k} U_j$.
Let us fix $n_0\in \mathbb{N}^{+}$ such that
$\frac{1}{N^{n_0}}<\min_{1\leq j\leq k} \diam U_j$.
For $j\in \{1,\dots,k\}$ let
%%%%
$$ t_j = \# \left\{(i_1,\dots,i_{n_0})\in \mathcal{I}_{n_0}:
U_j\cap K_{i_1\dots i_{n_0}}\neq \emptyset\right\} .$$
%%%
Since $A\subseteq \bigcup_{j=1}^{k}U_j$, we have
%%%%
\begin{equation} \label{eq:tj} \sum_{j=1}^{k} t_{j}\geq
(N-1)^{n_0}.
\end{equation}
%%%%%
Now we show that for all $j\in \{1,\dots,k\}$
%%%
\begin{equation} \label{eq:Uj} (\diam U_j)^{s} \geq
\frac{t_j}{2^{d}(N-1)^{n_0+1}}.
\end{equation}
%%%%%%
Let us fix $j\in \{1,\dots, k\}$. There exists $0\leq m<n_0$ such that
$\frac{1}{N^{m+1}}\leq \diam U_j < \frac{1}{N^{m}}$.
Clearly, the number of cubes $Q_{i_1 \dots i_m}$ at level $m$ that
intersect $U_j$ is at most $2^d$. Therefore $t_j\leq 2^{d} (N-1)^{n_0-m}$.
On the other hand, $\diam U_j\geq \frac{1}{N^{m+1}}$
implies $(\diam U_j)^{s}\geq \frac{1}{(N-1)^{m+1}}$,
and \eqref{eq:Uj} follows. By \eqref{eq:tj} and \eqref{eq:Uj} we obtain
%%%%%%
$$ \sum_{j=1}^{k} (\diam U_j)^{s} \geq
\sum_{j=1}^{k} \frac{t_j}{2^{d}(N-1)^{n_0+1}}\geq \frac{1}{2^d(N-1)}.$$
%%%%%
Hence $\mathcal{H}^{s} (A)>0$. Therefore $\dim_{H} A\geq s$,
so $s\leq \dim_{H}A \leq \overline{\dim}_{M}(A)\leq s$.
Thus $1-\varepsilon \leq \dim_{H} A=\overline{\dim}_{M}(A)<1$.

Now we are in a position to prove that there exists
a rectifiable curve $\Gamma$ with $\dim_{H}(\Gamma \cap K) = 1$.
Pick an arbitrary point $\vec{x} \in K$
and let $K_n$ be the intersection of $K$ and
the closed ball of radius $1/2^n$ centered at $\vec{x}$.
Let $C_n$ denote the component of $K_n$ containing $\vec{x}$.
Since $C_n$ is a non-degenerate continuum by Lemma \ref{l:comp}, we know that
there exists $A_n\subseteq C_n$ such that
$1-\frac{1}{n}\leq \dim_{H} A_n=\overline{\dim}_{M}(A)<1$.
Therefore Lemma \ref{l:cover} implies that
there exist rectifiable curves $\Gamma_n$ covering $A_n$.
We may assume that the endpoints of $\Gamma_n$ are in $A_n$.
We can also assume that the length of $\Gamma_n$ is at most $1/2^n$.
(Otherwise we split up $\Gamma_n$ into finitely many parts,
each having length at most $1/2^n$; then one of these parts
intersects $A_n$ in a set of Hausdorff dimension at least $1-\frac{1}{n}$.)
Let us concatenate the curves $\Gamma_n$ with line segments.
Then the full length of the line segments is
at most $2\sum_{n=1}^{\infty} \frac{1}{2^n}=2$,
the full length of the curves $\Gamma_n$ is at most
$\sum_{n=1}^{\infty} \frac{1}{2^n}=1$,
so we get a rectifiable curve $\Gamma$ that
covers $\bigcup_{n=1}^{\infty} A_n$.
As $\dim_{H}\left(\bigcup_{n=1}^{\infty} A_n\right)=1$,
the intersection $\Gamma \cap K$ has Hausdorff dimension $1$.
The proof is complete.
\end{proof}


\begin{thebibliography}{99}

\bibitem{E} R. Engelking, \textit{General topology}, Revised and
completed edition, Heldermann Verlag, 1989.

\bibitem{F} K. Falconer, \textit{Fractal geometry: Mathematical
foundations and applications}, Second Edition, John Wiley \& Sons,
2003.

\bibitem{G} M. Gromov, \textit{Partial differential relations},
Springer-Verlag, 1986.

\bibitem{BK} B. Kirchheim, private communication, 2011.

\bibitem{Ma} P. Mattila, \textit{Geometry of sets and measures in
Euclidean spaces}, Cambridge Studies in Advanced Mathematics No.~44,
Cambridge University Press, 1995.

\end{thebibliography}
\end{document}